\documentclass[12pt]{amsart}

\usepackage{amsmath}
\usepackage{amsmath}
\usepackage{amssymb}
\usepackage{verbatim}
\usepackage{graphicx}
\usepackage{amscd}
\usepackage{color}
\usepackage[latin1]{inputenc}
\usepackage{amsfonts}
\usepackage{latexsym}
\usepackage{epsfig}
\usepackage{ulem}
\newtheorem{theorem}{Theorem}[section]
\newtheorem{corollary}[theorem]{Corollary}

\newtheorem{proposition}[theorem]{Proposition}

\theoremstyle{definition}
\newtheorem{definition}[theorem]{Definition}

\newtheorem{remark}[theorem]{Remark}

\definecolor{purple}{rgb}{0.75, 0, 0.75}

\begin{document}
\title[Factoring onto $\mathbb{Z}^d$ subshifts with the fin. ext. property]{Factoring onto $\mathbb{Z}^d$ subshifts with the finite extension property}

\begin{abstract}
We define the finite extension property for $d$-di\-men\-sio\-nal subshifts, which generalizes the topological strong spatial mixing condition defined in \cite{briceno}, and we prove that this property is invariant under topological conjugacy. Moreover, we prove that for every $d$, every $d$-dimensional block gluing subshift factors onto every $d$-dimensional SFT with strictly lower entropy, a fixed point, and the finite extension property. This result extends a theorem from \cite{BPS}, which requires that the factor contain a safe symbol.
\end{abstract}

\date{}
\author{Raimundo Brice\~{n}o, Kevin McGoff, and Ronnie Pavlov}
\address{Raimundo Brice\~{n}o\\
School of Mathematical Sciences\\
Tel Aviv University\\
Tel Aviv 69978, Israel}
\email{raimundob@mail.tau.ac.il}
\urladdr{https://www.math.tau.ac.il/~raimundob/}
\address{Kevin McGoff\\
Department of Mathematics\\
University of North Carolina at Charlotte\\
Charlotte, NC 28223}
\email{kmcgoff1@uncc.edu}
\urladdr{https://clas-math.uncc.edu/kevin-mcgoff/}
\address{Ronnie Pavlov\\
Department of Mathematics\\
University of Denver\\
2390 S. York St.\\
Denver, CO 80208}
\email{rpavlov@du.edu}
\urladdr{http://www.math.du.edu/$\sim$rpavlov/}
\thanks{The first author acknowledges the support of ERC Starting Grants 678520 and 676970. The second author acknowledges the support of NSF grant DMS-1613261. The third author acknowledges the support of NSF grant DMS-1500685.}
\renewcommand{\subjclassname}{MSC 2010}
\maketitle


\section{Introduction}\label{intro}

A long-standing problem in the study of topological dynamical systems is the conjugacy problem, i.e., the problem of determining whether two dynamical systems which appear different actually exhibit the same dynamical behavior. A related problem is to determine when a topological dynamical system factors onto another one, i.e., when there is a surjective continuous map from the first to the second which intertwines their actions. Such maps are called (topological) factor maps, and they have been widely studied. We focus on these problems in the context of symbolic dynamical systems, also called subshifts.

For any natural number $d$ and finite set $\mathcal{A}$ (given the discrete topology), a $\mathbb{Z}^d$ subshift is any closed subset (with respect to the product topology) of $\mathcal{A}^{\mathbb{Z}^d}$ which is invariant under every translation $\sigma_t$ by a vector $t \in \mathbb{Z}^d$. 
We often refer to a subshift by the set $X$, with the understanding that the dynamics are always provided by the restriction of $\sigma$ to $X$. Examples of easily defined subshifts are the so-called $\mathbb{Z}^d$ shifts of finite type (or $\mathbb{Z}^d$ SFTs): for any finite set $\mathcal{F}$ of finite patterns, $X(\mathcal{F})$ is defined as the set of all elements of $\mathcal{A}^{\mathbb{Z}^d}$ which do not contain any pattern in $\mathcal{F}$. A special case is $X(\mathcal{\varnothing}) = \mathcal{A}^{\mathbb{Z}^d}$, called the full shift. 

There are two well-known necessary conditions for the existence of a factor map $\phi$ from $X$ onto $Y$. First, note that if $\sigma_t(x) = x$ for some $x \in X$ and $t \in \mathbb{Z}^d$, then $\sigma_t(\phi(x)) = \phi(x)$.   
Thus, $X$ and $Y$ must satisfy Condition (P): for every $x \in X$, there exists $y \in Y$ such that if $\sigma_t(x) = x$, then $\sigma_t(y) = y$.
Note that this condition is always satisfied when $Y$ contains a fixed point, i.e. $y \in Y$ where $\sigma_t(y) = y$ for all $t \in \mathbb{Z}^d$.
Second, the topological entropy of a $\mathbb{Z}^d$ subshift $X$ 
(denoted by $h(X)$; see Section~\ref{defs} for the definition) cannot increase under a factor map, and so $h(X) \geq h(Y)$ must hold. Surprisingly, for restricted classes of subshifts, these necessary conditions also seem to be nearly sufficient. (A stronger form of the following theorem appears in \cite{boyle}.) 

\begin{theorem}[\cite{boyle}] 
For mixing $\mathbb{Z}$ SFTs $X$ and $Y$ with $h(X) > h(Y)$, there exists a factor map from $X$ onto $Y$ if and only if $X$ and $Y$ satisfy Condition (P). 
\end{theorem}

When $d=1$ and $Y$ is a full shift, even the equal entropy case ({i.e., $h(X) = h(Y)$) has been solved. In this case, $Y$ automatically contains a fixed point, and so no additional periodic point hypothesis is necessary.

\begin{theorem}[\cite{boyle}, \cite{marcus}] 
For a $\mathbb{Z}$ SFT $X$ and a full shift $Y$ with $h(X) \geq h(Y)$, there exists a factor map from $X$ onto $Y$.
\end{theorem}

Unfortunately, the situation is much more complicated for $d > 1$. In particular, there are several different candidates for a proper extension of ``mixing'' to the multidimensional case. One commonly used condition is the block gluing condition defined in \cite{BPS}, and a much stronger one is the existence of a so-called safe symbol (definitions are given in Section~\ref{defs}). We do not attempt to summarize the entire literature on this topic, but here are a few representative results. First, the theorems for $\mathbb{Z}$ subshifts do not directly extend to $\mathbb{Z}^d$ subshifts when $d > 1$.

\begin{theorem}[\cite{BPS}] 
\label{theoremBPS1}
For every $d > 1$, there exist topologically mixing $\mathbb{Z}^d$ SFTs with arbitrarily high entropy which do not factor onto any nontrivial full shift.
\end{theorem}

\begin{theorem}[\cite{PS}]
\label{theoremPS}
For every $d \geq 3$ and every nontrivial $\mathbb{Z}^d$ full shift $Y$, there exists a block gluing $\mathbb{Z}^d$ SFT $X$ with $h(X) = h(Y)$ such that there is no factor map from $X$ onto $Y$.
\end{theorem}


Under a strict entropy inequality, the block gluing hypothesis, which allowed for the negative examples of Theorem~\ref{theoremPS}, implies a positive result for $d > 1$ even for general subshifts.



\begin{theorem}[\cite{BPS}]
\label{theoremBPS2}
If $X$ is a block gluing $\mathbb{Z}^d$ subshift, $Y$ is a $\mathbb{Z}^d$ SFT with a safe symbol, and $h(X) > h(Y)$, then there exists a factor map from $X$ onto $Y$.
\end{theorem}

We also note that the safe symbol hypothesis in Theorem \ref{theoremBPS2} is very restrictive, and is not at all invariant under topological conjugacy.



In this work, we define a new condition called the finite extension property, which is significantly weaker than the existence of a safe symbol. We prove that this condition is conjugacy-invariant, and then we prove the following main result.

\begin{theorem}\label{mainthm}
If $X$ is a block gluing $\mathbb{Z}^d$ subshift, $Y$ is a $\mathbb{Z}^d$ SFT with a fixed point and the finite extension property, and $h(X) > h(Y)$, then there exists a factor map from $X$ onto $Y$.
\end{theorem}


For a $\mathbb{Z}^d$ SFT defined by a set of forbidden pairs of adjacent letters, an easily verified (but not conjugacy invariant) condition is single-site fillability or SSF (\cite{MP}). For $d=2$, SSF means 
that for any choice of letters $a,b,c,d \in \mathcal{A}$, there exists $e \in \mathcal{A}$ for which the pattern $\begin{smallmatrix} & a & \\ b & e & c\\ & d & \end{smallmatrix}$ contains none of the forbidden adjacent pairs. Using the forbidden adjacencies as the set of forbidden patterns, it is straightforward to check that SSF implies the finite extension property. The following corollary is immediate.

\begin{corollary}\label{maincor}
If $X$ is a block gluing $\mathbb{Z}^2$ subshift, $Y$ is a $\mathbb{Z}^2$ SFT that satisfies single-site fillability and has a fixed point, and $h(X) > h(Y)$, then there exists a factor map from $X$ onto $Y$.
\end{corollary}

Corollary~\ref{maincor} can be used to create explicit examples of new subshifts to which our results apply, since there are many nearest-neighbor $\mathbb{Z}^d$ SFTs which have fixed points and satisfy SSF without having a safe symbol. For instance, one can take any alphabet $\mathcal{A}$ with $|\mathcal{A}| \geq 2d+1$, take any non-identity involution $f$ on $\mathcal{A}$, and define $Y$ by the rule that no pair of letters $\{a,f(a)\}$ ($a \in \mathcal{A}$) can be adjacent.

\section*{acknowledgements} 
The authors would like to thank the anonymous referee for making many useful comments, which significantly improved the clarity and presentation of this work.



\section{Definitions}\label{defs}

We begin with some basic geometric definitions for $\mathbb{Z}^d$. Anytime we refer to distance in $\mathbb{Z}^d$, it is with respect to the $\ell_{\infty}$ distance given by $d((v_i)_{i=1}^d, (w_i)_{i=1}^d) = \max_i(|v_i - w_i|)$. 
For sets $A, B \subset \mathbb{Z}^d$, we define $d(A,B) = \min_{a \in A, b \in B} d(a,b)$. For every $k$, we use $C_k$ and $Q_k$ to denote the hypercubes $[0,k-1]^d$ and $[-k,k]^d$ respectively. For any set $S \subset \mathbb{Z}^d$, we define its \textbf{inner $k$-boundary} $\partial_k S$ to be the set of all $t \in S$ within distance $k$ from some $t' \in S^c$.


\begin{definition}
A \textbf{pattern} over a finite alphabet $\mathcal{A}$ is a member of $\mathcal{A}^S$ for some $S \subset \mathbb{Z}^d$, which is said to have \textbf{shape} $S$. We may refer to any pattern with finite shape as a \textbf{finite pattern}.
\end{definition}

We consider patterns to be defined up to translation: if $u \in \mathcal{A}^S$ for a finite $S \subset \mathbb{Z}^d$ and $v \in \mathcal{A}^T$, where $T = S+t$ for some $t \in \mathbb{Z}^d$, then we write $u = v$ to mean that $u(s) = v(s+t)$ for each $s$ in $S$. 

For any patterns $v \in \mathcal{A}^S$ and $w \in \mathcal{A}^T$ with $S \cap T = \varnothing$, we define the concatenation $vw$ to be the pattern in $\mathcal{A}^{S \cup T}$ defined by $(vw)(S) = v$ and $(vw)(T) = w$.

\begin{definition}
For any finite alphabet $\mathcal{A}$, the \textbf{$\mathbb{Z}^d$-shift action} on $\mathcal{A}^{\mathbb{Z}^d}$, denoted by $\{\sigma_t\}_{t \in \mathbb{Z}^d}$, is defined by $(\sigma_t x)(s) = x(s+t)$ for $s,t \in \mathbb{Z}^d$. 
\end{definition}

We always think of $\mathcal{A}^{\mathbb{Z}^d}$ as being endowed with the product discrete topology, with respect to which it is compact. 

\begin{definition}
A \textbf{$\mathbb{Z}^d$ subshift} is a closed subset of $\mathcal{A}^{\mathbb{Z}^d}$ that is invariant under the $\mathbb{Z}^d$-shift action. 
\end{definition}

Any $\mathbb{Z}^d$ subshift inherits a topology from $\mathcal{A}^{\mathbb{Z}^d}$, with respect to which it is compact. Each $\sigma_t$ is a homeomorphism on any $\mathbb{Z}^d$ subshift, and so any $\mathbb{Z}^d$ subshift, when paired with the $\mathbb{Z}^d$-shift action, is a topological dynamical system. 

Any $\mathbb{Z}^d$ subshift can also be defined in terms of forbidden patterns: for any set $\mathcal{F}$ of finite patterns over $\mathcal{A}$, one can define the set 
$$X(\mathcal{F}) := \{x \in \mathcal{A}^{\mathbb{Z}^d} \ : \ x(S) \notin \mathcal{F} \ \text{ for all finite } S \subset \mathbb{Z}^d\}.$$It is well known that any set of the form $X(\mathcal{F})$ is a $\mathbb{Z}^d$ subshift, and all $\mathbb{Z}^d$ subshifts may be presented in this way. 

\begin{definition} 
A \textbf{$\mathbb{Z}^d$ shift of finite type (SFT)} is a $\mathbb{Z}^d$ subshift equal to $X(\mathcal{F})$ for some finite set $\mathcal{F}$ of forbidden finite patterns. 
\end{definition}

\begin{definition} 
The \textbf{language} of a $\mathbb{Z}^d$ subshift $X$, denoted by $L(X)$, is the set of all patterns that appear in elements of $X$. For any $S \subset \mathbb{Z}^d$, let $L_S(X) := L(X) \cap \mathcal{A}^S$, the set of patterns in the language of $X$ with shape $S$. A finite pattern $w$ will be called a \textbf{first offender} for $X$ if it is not in $L(X)$ but every proper subpattern of $w$ belongs to $L(X)$.
\end{definition}

\begin{remark}
We have defined the language of a subshift to include both the finite and infinite patterns that appear in elements of $X$. We adopt this convention for convenience of presentation, despite the fact that many authors do not include infinite patterns in the language.
\end{remark}

\begin{definition}
Suppose $X$ and $Y$ are compact, metrizable spaces. Further suppose that $\mathbb{Z}^d$ acts on each of these spaces by homeomorphisms, with actions denoted by $\sigma$ and $\tau$, respectively.
A (topological) \textbf{factor map} is any continuous surjection $\phi : X \to Y$ such that $\phi \circ \sigma_t = \tau_t \circ \phi$ for each $t \in \mathbb{Z}^d$. In this case, the pair $(Y,\tau)$ is called a \textbf{factor} of $(X,\sigma)$, and we say that $X$ factors onto $Y$. A bijective factor map is called a \textbf{topological conjugacy}.
\end{definition}

For the purposes of this work, we restrict attention to factor maps between subshifts. It is well-known that any factor map $\phi$ between $\mathbb{Z}^d$ subshifts is a so-called sliding block code, i.e., there exists $n \in \mathbb{N}$ so that $x(t + [-n,n]^d)$ uniquely determines $(\phi(x))(t)$ for any $x \in X$ and $t \in \mathbb{Z}^d$; such $n$ is usually called a radius for the sliding block code. (See \cite{LM} for a proof for $d = 1$, which extends to $d > 1$ without changes.) When convenient, for a pattern $w$ with shape $S$, we may use $\phi(w)$ to denote its image under a sliding block code $\phi$ with radius $n$, with shape $S \setminus \partial_n S$.



\begin{definition}\label{entdef}
The \textbf{topological entropy} of a $\mathbb{Z}^d$ subshift $X$ is
$$h(X) := \lim_{n \rightarrow \infty} \frac{1}{n^d} \log | L_{C_n}(X) |.$$
This limit exists by a standard subadditivity argument.
\end{definition}


Finally, let us define the mixing properties for $\mathbb{Z}^d$ subshifts which we will need.



\begin{definition}\label{block}
A $\mathbb{Z}^d$ subshift $X$ is {\bf block gluing} if there exists $g \geq 0$ so that for any hyperrectangles $R, R' \subset \mathbb{Z}^d$ with $d(R, R') > g$ and any $w \in L_R(X)$ and $w' \in L_{R'}(X)$, there exists $x \in X$ with $x(R) = w$ and $x(R') = w'$.
\end{definition}

\begin{definition}\label{safe}
A letter $* \in \mathcal{A}$ is a \textbf{safe symbol} for a $\mathbb{Z}^d$ subshift $X$ if for any point $x \in X$ and any $S \subseteq \mathbb{Z}^d$, changing each letter of $x$ on $S$ to $*$ yields a point in $X$.
\end{definition}


\begin{definition}\label{extprop}
For $g \in \mathbb{N}$, a $\mathbb{Z}^d$ SFT $X$ has the \textbf{$g$-extension property} if there exists a finite set $\mathcal{F}$ of forbidden finite patterns inducing $X$ with the following property: if a pattern $w$ with shape $S$ can be extended to a pattern on $S + Q_g$ which does not contain any patterns from $\mathcal{F}$, then $w \in L(X)$, i.e., it can be extended to a point on all of $\mathbb{Z}^d$ which does not contain any patterns from $\mathcal{F}$. We say that $X$ has the \textbf{finite extension property} if it has the $g$-extension property for some $g$.
\end{definition}

(The reader may check that any $X$ with the $g$-extension property is block gluing at distance $2g$ plus the maximum diameter over $w \in \mathcal{F}$.)

The \textbf{topological strong spatial mixing (TSSM) property} for $\mathbb{Z}^d$ SFTs was introduced in \cite{briceno}, where it was also shown to be equivalent to the existence of only finitely many first offenders for $X$.

\begin{proposition}
A $\mathbb{Z}^d$ SFT $X$ has the TSSM property if and only if it has the $0$-extension property.
\end{proposition}

\begin{proof}
Suppose that $X$ has the TSSM property and therefore has only finitely many first offenders. Let $\mathcal{F}$ denote the list of first offenders. We claim that $X$ has the $0$-extension property for $\mathcal{F}$. In fact, $X = X(\mathcal{F})$ and if $w$ is a pattern not in $L(X)$, then $w$ must contain a minimal subpattern not in $L(X)$, which by definition is a first offender.

For the reverse implication, suppose that $X$ has the $0$-extension property for a finite set $\mathcal{F}'$ of forbidden finite patterns of diameter at most $g$. Assume, for the sake of contradiction, that $w$ is a first offender of diameter greater than $g$. Then, $w \notin \mathcal{F}'$ and, by definition of first offender, every proper subpattern of $w$ is in $L(X)$ and so not in $\mathcal{F}'$. Therefore, by the $0$-extension property, $w$ is in $L(X)$, contradicting the assumption that $w$ is a first offender. We conclude that first offenders have bounded diameter, so there must be finitely many of them.
\end{proof}

It is known that the existence of a safe symbol implies TSSM (see \cite{briceno}). Thus we have the following corollary.

\begin{corollary}
If $X$ is a $\mathbb{Z}^d$ SFT with a safe symbol, then $X$ has the $0$-extension property.
\end{corollary}


As noted in the introduction, the finite extension property is also invariant under topological conjugacy.

\begin{theorem}\label{conjinv}
If $X$ and $Y$ are conjugate $\mathbb{Z}^d$ SFTs and $X$ has the finite extension property, then $Y$ has the finite extension property.
\end{theorem}

\begin{proof}
Suppose that $X$ has the $g$-extension property (for forbidden list $\mathcal{F}$) and that $\phi: X \rightarrow Y$ is a conjugacy. Denote by $r$ the radius of $\phi$ and by $s$ the radius of $\phi^{-1}$. Define a list of patterns on $\mathcal{A}_Y$ as follows:
\[
\mathcal{F}' := \{w \in \mathcal{A}_Y^{S + Q_s} \ : \ v \in \mathcal{F}, \textrm{ $v$ has shape $S$, } \phi^{-1}(w) \textrm{ contains } v\}.
\]

Clearly $\mathcal{F}'$ is a finite list of finite patterns, and we claim that it induces the shift of finite type $Y$. 
Indeed, by definition, if $y \in Y$, then $\phi^{-1}(y) \in X$, and therefore $y$ contains no pattern in $\mathcal{F}'$. On the other hand, if $y \in \mathcal{A}_Y^{\mathbb{Z}^d}$ contains no pattern in $\mathcal{F}'$, then the point $x$ defined by $x(t) = \phi^{-1}(y(t+Q_s))$ contains no pattern in $\mathcal{F}$, so is in $X$, and therefore $y = \phi(x)$ is in $Y$.
 
Now assume that a pattern $w \in \mathcal{A}_Y^S$ can be extended to a pattern $v \in \mathcal{A}_Y^{S + Q_{g+r+s}}$ containing no patterns from $\mathcal{F'}$. Then, by definition, $\phi^{-1}(v)$ contains no patterns from $\mathcal{F}$; say that $\phi^{-1}(v)$ has shape $T$, and note that $T \supseteq S + Q_{g+r}$. Then by $g$-extension of $X$, the pattern $(\phi^{-1}(v))(T \setminus \partial_g(T))$ is in $L(X)$. Then obviously 
$\phi((\phi^{-1}(v))(T \setminus \partial_g(T))) \in L(Y)$, and we note that its shape contains $S$. Finally, by definitions of $r$ and $s$, we have $\phi((\phi^{-1}(v))(T \setminus \partial_g(T)))(S) = w$, and so $w \in L(Y)$, completing the proof. 
\end{proof}

\section{Proof of Theorem~\ref{mainthm}}\label{proof}

The overall structure of our proof is similar to previous proofs which used 
mixing properties to construct factor maps onto various shifts (see \cite{BPS} and \cite{desai}). By this, we mean that the proof involves using marker patterns to define ``surrounded patterns'' in points of the domain, which will be used to assign patterns on ``determined zones'' after application of the map. 
Very roughly speaking, given $x \in X$, its image $\phi(x)$ will have patterns on determined zones that depend on corresponding surrounded patterns in $x$, and $\phi(x)$ will look like the fixed point of $Y$ at all sites not near a determined zone. Then we will fill the area between determined zones and the fixed point ``background'' in stages using the $g$-extension property of $Y$. First we give the proof for $d = 2$ in order to present a streamlined argument with illustrations, and then we describe the changes that need to be made for $d > 2$.

To begin the formal proof, choose any $X$ and $Y$ as in the theorem, with alphabets $\mathcal{A}_X$ and $\mathcal{A}_Y$, respectively. We assume without loss of generality that $g \geq 0$ is a gap distance for the block gluing of $X$, that $Y$ has the $g$-extension property for a finite list $\mathcal{F}$ of forbidden finite patterns with diameters less than or equal to $g$, 
and that the fixed point $*^{\mathbb{Z}^2}$ is in $Y$. 

We now construct markers in $X$ following \cite{BPS}, but we repeat some details here to set notation. Let $p > 5g$, and choose a pattern $P \in L_{C_p}(X)$ so that $h(X_P) > h(Y)$, where $X_P$ is the subshift consisting of points of $X$ which do not contain the pattern $P$ (see \cite{QT}). Then define a pattern $Q \in L_{C_q}(X_P)$ (for some $q \in \mathbb{N}$ perhaps much larger than $p$) for which $Q$ cannot overlap itself at any nonzero vector in $Q_{g+p} = [-g-p,g+p]^2$, i.e., for every such vector $t$, there does not exist $x \in \mathcal{A}^{\mathbb{Z}^2}$ for which $x(C_q) = x(C_q + t) = Q$ (see \cite{BPS,desai}). Then use block gluing to create a marker pattern $M \in L_{C_m}(X)$ ($m = 2p + 2g + q$) with $P$ at each corner, $Q$ in the center, and patterns $G_i \in L(X_P)$, $1 \leq i \leq 4$, along each edge, as in the left half of Figure~\ref{TSSMpic1}. Any pattern as in the right half of Figure~\ref{TSSMpic1}, where $W \in L_{C_k}(X_P)$ and each $H_i \in L(X_P)$, $1 \leq i \leq 4$, is called a \textbf{surrounding frame}, whose central occurrence of $W$ is called a \textbf{surrounded pattern}. 
The side length $k$ of the shape of $W$ is for now arbitrary, and will be fixed later. For any surrounding frame $x(t+C_{k+2g+2m})$ in $x$, we refer to the region $t + (g+m) \vec{1} + C_{k+g+m}$ as a \textbf{determined zone} in $\phi(x)$. 

\begin{figure}[ht]
\centering
\includegraphics[scale=0.6125]{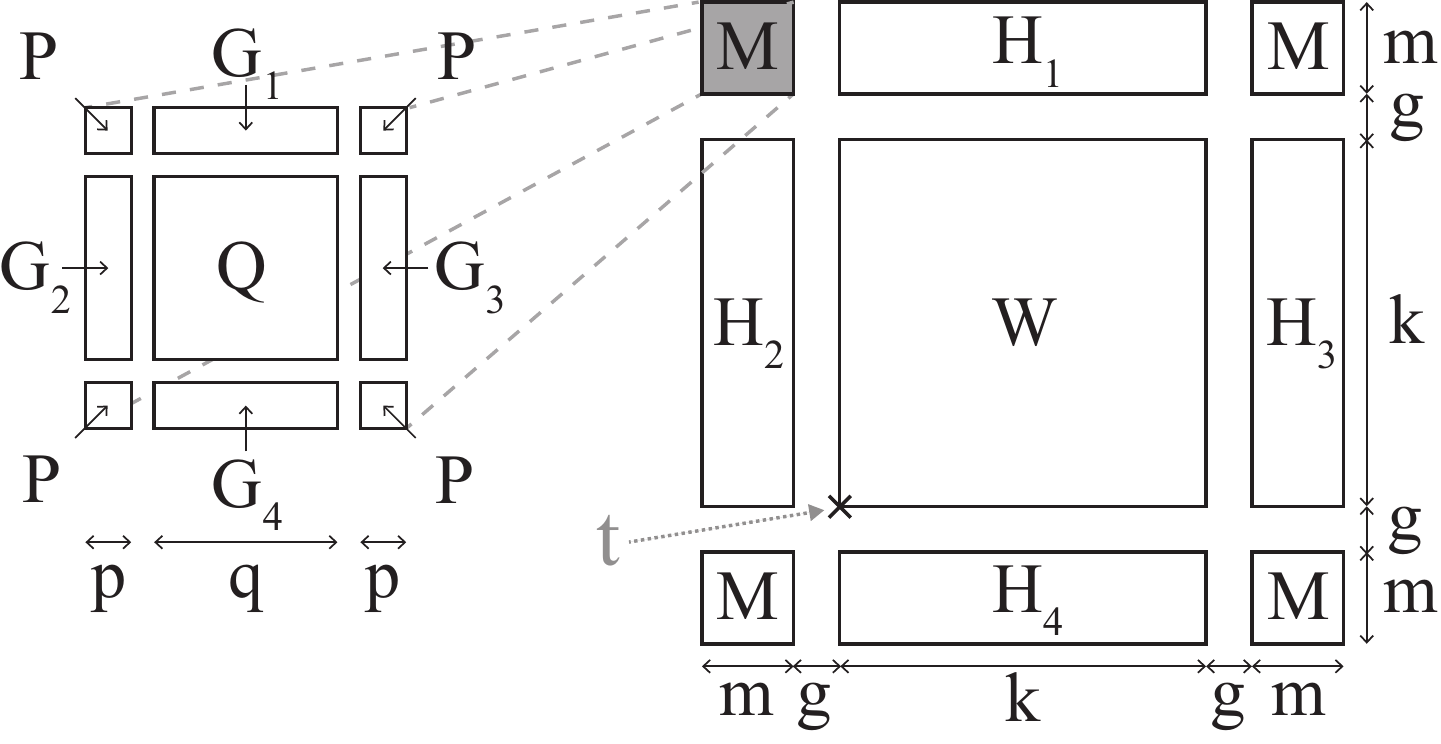}
\caption{A marker pattern (left) and a surrounding frame (right)}
\label{TSSMpic1}
\end{figure}

We need a few simple facts about the locations of determined zones. Firstly, shown exactly as in \cite{BPS} and \cite{desai}, by the marker properties defining $M$, any two determined zones have distance more than $g$ from each other. In fact, for any two determined zones $t_1 + C_{k+g+m}$ and $t_2 + C_{k+g+m}$ with distance exactly $g+1$, the surrounding frames $x(t_1 - (g+m) \vec{1} + C_{k+2g+2m})$ and $x(t_2 - (g+m) \vec{1} + C_{k+2g+2m})$ have overlap consisting of either exactly one occurrence of $M$ or a rectangle with dimensions $m$ and $k + 2g + 2m$ with occurrences of $M$ at the extreme ends. (See Figure~\ref{TSSMpic2}.) In either case, we say that those determined zones are \textbf{adjacent}. We use the term \textbf{component of determined zones} to refer to a maximal connected component with respect to this notion of adjacency. Finally, we claim that if two determined zones $Z_1$ and $Z_2$ are not adjacent, then
\begin{equation} \label{Eqn:Banff}
d(Z_1,Z_2) > 2g+p > 7g.
\end{equation} 

To see this, suppose for a contradiction that two determined zones are separated by distance more than $g$ and less than or equal to $2g + p$. This means that $x$ contains two surrounding frames separated by a vector $t = (t_1, t_2)$ where $k + 2g + m < \max(|t_1|, |t_2|) \leq k + 3g + m + p$, which without loss of generality we can take to be $x(C_{k+3g+m})$ and $x(t + C_{k + 3g + m})$. Our argument will rely only on the general structure of surrounded frames (and not the specific values of $W$ or the $H_i$), and so is unaffected by reflections about horizontal, vertical, or diagonal lines. Therefore, we may assume without loss of generality that $k + 2g + m < t_1 \leq k + 3g + m + p$ and $0 \leq t_2 \leq k + 3g + m + p$. 

We note that if $t_2 \in [0, g+p]$, then $d(t, (k+2g+m,0)) \leq g + p$, meaning that the lower-right copy of $M$ within $x(C_{k+2g+2m})$ and the lower-left copy of $M$ within $x(t + C_{k+3g+m})$ would have separation by a nonzero vector in $Q_{g+p}$. This contradicts the definition of $Q$ and so is impossible. The case $t_2 \in [k+2g+m, k+3g+m+p]$ is also not possible, by a similar argument using the upper-right copy of $M$ within $x(C_{k+2g+2m})$ and the lower-left copy of $M$ within $x(t + C_{k+3g+m})$. Therefore, $t_2 \in (g+p, k+2g+m)$. However, this implies that the lower-left copy of $M$ within $(t + C_{k + 2g + 2m})$ overlaps the pattern $H$ along the right side of $x(C_{k+2g+2m})$ in a rectangle with height at least $p$ and width at least $m - g - p = p + g + q$. This yields a contradiction since $M$ has a copy of $P$ in each corner and $H$ was assumed in $L(X_P)$. We have thus established (\ref{Eqn:Banff}), a fact which will be useful later. \\
\begin{figure}[ht]
\centering
\includegraphics[scale=0.6125]{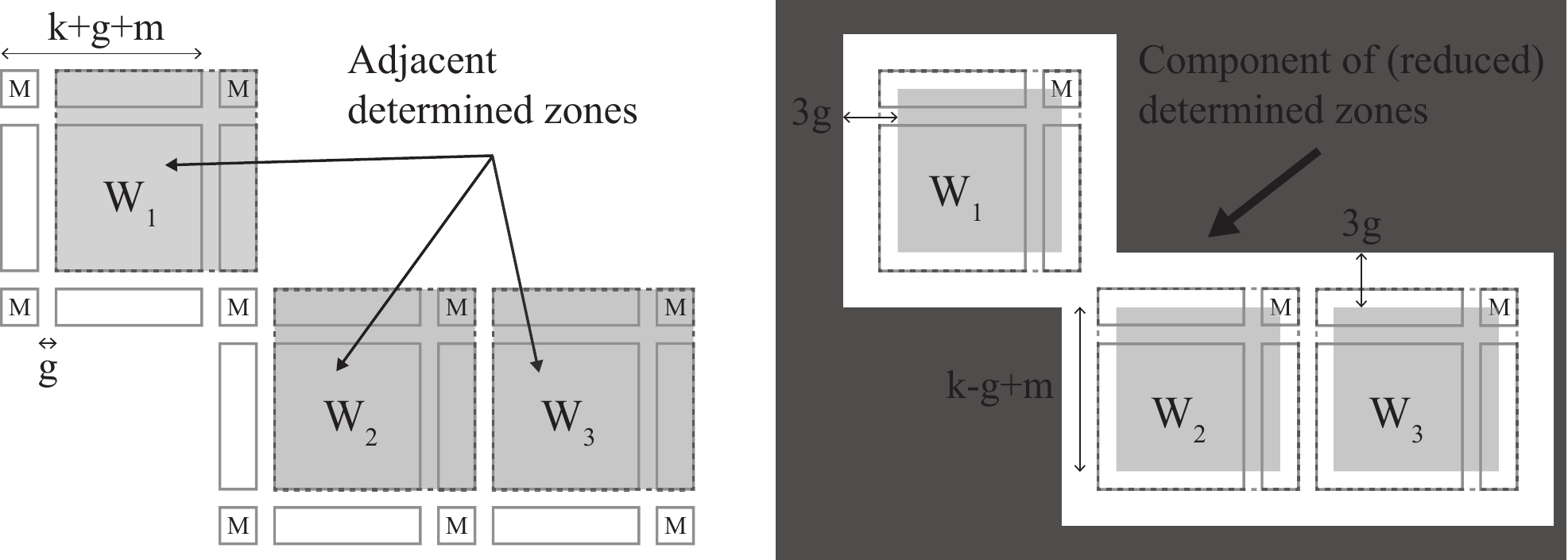}
\caption{Surrounding frames and component of determined zones induce by them (left). Reduced determined zones after Stage 3 in a $*$-background (right).}
\label{TSSMpic2}
\end{figure}
Now let $x$ be in $X$. Informally speaking, $\phi(x)$ will be defined in six alternating stages, determined completely by the surrounded patterns in $x$. After each odd-indexed stage $2i-1$ ($i = 1,2$), $\phi(x)$ will be defined on a set $U_{2i-1}$ as a pattern $u_{2i-1} \in L(Y)$. Then, the following (even-indexed) stage $2i$ will define $\phi(x)$ on a set $S_{2i}$, where $\phi(x)(S_{2i})$ is a pattern $s_{2i}$ for which $v_{2i} = u_{2i-1} s_{2i}$ on $V_{2i} := U_{2i-1} \sqcup S_{2i}$ contains no patterns from $\mathcal{F}$. The following (odd-indexed) stage $2i+1$ will remove all letters on $\partial_g V_{2i}$, yielding a pattern $u_{2i+1}$ on $U_{2i+1} := V_{2i} \setminus \partial_g V_{2i}$. Then $u_{2i+1} \in L(Y)$ by the $g$-extension property, allowing the process to continue. The patterns placed during even-indexed stages are dependent only on nearby surrounded patterns in $x$; to describe this dependency, we require the following auxiliary function.

Since $h(X_P) > h(Y)$, for sufficiently large $k$ it is the case that
\begin{equation}\label{d2bound}
|L_{C_k}(X_P)| > |L_{C_{k+g+m}}(Y)| \cdot |\mathcal{A}_Y|^{12g(k-3g+m) + 196g^2}.
\end{equation}
Fix any such $k$ (which does not depend on $x$), and then define a surjection $\psi$ from $L_{C_k}(X_P)$ to the set of all tuples of the form 
$(i_j)_{1 \leq j \leq 9}$, where $1 \leq i_1 \leq |L_{C_{k+g+m}}(Y)|$, $1 \leq i_j \leq |\mathcal{A}_Y|^{3g(k-3g+m)}$ for $2 \leq j \leq 5$, and $1 \leq i_j \leq |\mathcal{A}_Y|^{49g^2}$ for $6 \leq j \leq 9$. We are now ready to describe the stages of defining the factor map $\phi$.\\

\textbf{Stage 1:} Define $U_1$ to be the set of all $t \in \mathbb{Z}^2$ at a distance of more than $g$ from all determined zones, and define $u_1 = *^{U_1}$. Clearly $u_1 \in L(Y)$ since $*^{\mathbb{Z}^2} \in Y$. We note that after Stage 1, the undefined portion of $\phi(x)$ consists of components of determined zones, along with all sites within distance $g$ of them; we use the term ``island'' to denote the set of sites within distance $g$ of such a component. By (\ref{Eqn:Banff}), any two nonequal islands have distance more than $5g$. 
For any island $I$, and for $i = 1,2$, define $T_i(I)$ to be the sets of $e_1$- and $e_2$-coordinates (respectively) which appear in some determined zone in $I$.\\

\textbf{Stage 2:} For each island $I$, the set $I \cap (T_1(I) \times T_2(I))$ is the disjoint union of the determined zones in the component inducing $I$. Let $S_2 = \bigcup_I (I \cap (T_1(I) \times T_2(I)))$. We define a pattern $s_2$ on $S_2$ as follows. For any determined zone $t + C_{k+g+m}$, by definition $x(t + C_k)$ is a surrounded pattern in $x$. Let the tuple $(i_j)_{1 \leq j \leq 9}$ be defined by $\psi(x(t+C_k)) = (i_j)_{1 \leq j \leq 9}$, and then let $s_2(t+C_{k+g+m})$ be the $i_1$th pattern in $L_{C_{k+g+m}}(Y)$ according to the lexicographic ordering. Then $s_2$ is just the concatenation of these patterns.

We define $V_2 = U_1 \sqcup S_2$ and $v_2 := u_1 s_2$. Each pattern placed on a determined zone was assumed to be in $L(Y)$, and so contained no patterns from $\mathcal{F}$. As noted above, the same is true for the $*$-pattern $u_1$ placed on $U_1$.} Since patterns in $\mathcal{F}$ have diameters less than $g$ and since determined zones have distance greater than $g$ from each other and from $U_1$, $v_2$ contains no patterns from $\mathcal{F}$.\\

\textbf{Stage 3:} Define $U_3 = V_2 \setminus \partial_g V_2$, and $u_3 := v_2(U_3)$. By the $g$-extension property, $u_3 \in L(Y)$. 

To more easily describe future stages, we describe the structure of the set $U_3$. Namely, $U_3$ consists of two types of sites: those at distance more than $2g$ from all determined zones, and those within a determined zone in an island $I$ and for which both coordinates are at distance more than $g$ from the corresponding $T_i(I)^c$. For each island $I$, $U_3 \cap I$ consists of a disjoint union of squares obtained from removing the inner $g$-boundary from each determined zone; we call these squares ``reduced determined zones.'' (See Figure~\ref{TSSMpic2}.) \\

\textbf{Stage 4:} Define $S_4$ to be the set of all sites which are within distance $2g$ of some determined zone in an island $I$, 
have one coordinate which is within distance $g$ of the corresponding $T_i(I)^c$, and one coordinate which has a distance of more than $2g$ from the corresponding $T_i(I)^c$. Informally, $S_4$ is the (disjoint) union of all rectangles with dimensions $3g$ and $k-3g+m$ that share (at least one of) their longest side(s) with a reduced determined zone and are centered along the corresponding side of that reduced determined zone. 
Any two such rectangles are separated by distance greater than $g$; if they're part of the same island then this is true since reduced determined zones have side length greater than $m - g > 2p - g > g$, and if they are part of different islands then this follows from (\ref{Eqn:Banff}).

We now define a pattern $s_4$ on $S_4$. Choose any of the rectangles $R$ comprising $S_4$. First, we need a way to associate a determined zone to $R$; to this end, choose the first direction in the ordering $\{$up, left, down, right$\}$ for which there is a reduced determined zone adjacent to $R$ in that direction, which came from some determined zone. Since $u_3 \in L(Y)$, there exists a pattern on $R$ which yields a pattern in $L(Y)$ when concatenated with $u_3$. However, we need to choose such a pattern on $R$ using only the portion of $x$ which lies within a uniformly bounded distance of $R$ to ensure that $\phi$ is a sliding block code, and if the island $I$ is quite large, then there is no obvious way to do so. Instead, we settle for choosing a pattern on $R$ which creates no patterns from $\mathcal{F}$ when concatenated with $u_3$. That is, consider the collection of patterns $\{w \in (\mathcal{A}_Y)^R :  u_3 w \textrm{ contains no patterns from } \mathcal{F}\}$; note that this collection depends only on the portion of $u_3$ within distance $g$ of $R$. Since $u_3$ was in $L(Y)$, this collection is nonempty, and trivially, it has cardinality bounded from above by $|\mathcal{A}_Y|^{|R|} \leq |\mathcal{A}_Y|^{3g(k-3g+m)}$. 

We then define $s_4(R)$ to be the $i_j$th pattern in this collection according to the lexicographic ordering, where $t + C_{k+g+m}$ was the determined zone associated to $R$ above, $\psi(x(t + C_k)) = (i_j)_{1 \leq j \leq 9}$, and $j$ is taken to be $2$, $3$, $4$, or $5$ based on whether $t + C_{k+g+m}$ is reached by moving up, left, down, or right from $R$. (We adopt the convention, here and later, that for a totally ordered set $S$ and $n > |S|$, the $n$th element of $S$ is just taken to be the maximal element.) We note for future reference that no $i_j$ determines patterns on two different rectangles $R$. Now, $s_4$ is just the concatenation of these patterns.



Define $V_4 = U_3 \sqcup S_4$ and $v_4 = u_3 s_4$. (See Figure~\ref{TSSMpic8}.) No forbidden pattern in $\mathcal{F}$ can intersect two rectangles $R$ since distinct rectangles $R$ are separated by distance more than $g$. No forbidden pattern in $\mathcal{F}$ can intersect exactly one rectangle $R$ since $u_3 s_4(R)$ was assumed not to contain such patterns. Finally, no forbidden pattern in $\mathcal{F}$ can occur disjointly from all rectangles $R$ since $u_3 \in L(Y)$. Therefore, $v_4$ contains no patterns from $\mathcal{F}$.

\begin{figure}[ht]
\centering
\includegraphics[scale=0.6125]{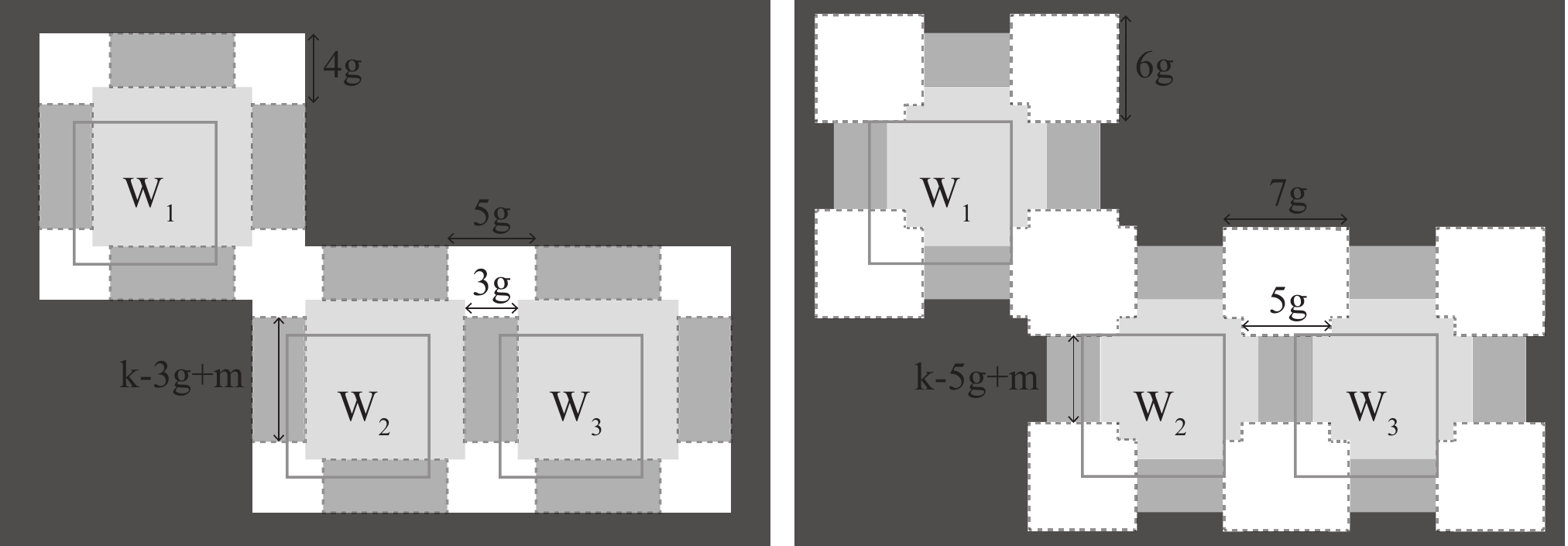}
\caption{Sites assigned during Stages 1,2, and 4 are in dark gray, light gray, and medium gray, respectively (left). Sites assigned during Stage 6 are in white and doubly reduced determined zones in light gray (right).}
\label{TSSMpic8}
\end{figure}

\textbf{Stage 5:} Define $U_5 = V_4 \setminus \partial_g V_4$, and $u_5 := v_4(U_5)$. By the $g$-extension property, $u_5 \in L(Y)$. 
Again we explicitly describe the structure of $U_5$. 
Now, $U_5$ consists of three types of sites. The first are those which are at a distance of more than $3g$ from all determined zones. The second are those which are within distance $3g$ from a determined zone in an island $I$, and for which both coordinates have distance more than $2g$ from the corresponding $T_i(I)^c$. Such sites form a disjoint union of squares obtained by removing the inner $2g$-boundary from all determined zones; we call these ``doubly reduced determined zones.'' 
The third type are those which are within distance $3g$ from a determined zone in an island $I$, have one coordinate within distance $2g$ from the corresponding $T_i(I)^c$, and one coordinate with distance more than $3g$ from the corresponding $T_i(I)^c$.\\


\textbf{Stage 6:} We define $S_6 = U_5^c$. From the description above, it should be clear that the sites in $S_6$ have the following properties: they are within distance $3g$ from a determined zone in an island $I$, have one coordinate within distance $2g$ from the corresponding $T_i(I)^c$, and the other coordinate within distance $3g$ from the corresponding $T_i(I)^c$. By (\ref{Eqn:Banff}), sites in $S_6$ associated to different islands have distance at least $g$. Since doubly reduced determined zones have side length greater than 
$m - 3g > 2p - 3g > 3g$, we see that $S_6$ consists of a disjoint union of connected components with diameters at most $7g$ separated by distance more than $g$, which we call holes. 

We fill the holes with patterns in much the same way as in Stage 4. We again associate a determined zone to each hole $H$; to this end, choose the first direction in the ordering $\{$up-left, up-right, down-left, down-right$\}$ for which there is a doubly reduced determined zone adjacent to $H$ in that direction, which came from some determined zone.

For each hole $H$, consider the collection of patterns 
$\{w \in (\mathcal{A}_Y)^H \ : \ u_5 w \textrm{ contains no patterns from } \mathcal{F}\}$. Since $u_5 \in L(Y)$, this collection is nonempty, and its cardinality is at most $|\mathcal{A}_Y|^{|H|} \leq |\mathcal{A}_Y|^{49g^2}$. 

We define $s_6(H)$ to be the $i_j$th pattern in this collection according to the lexicographic ordering, where 
$t + C_{k+g+m}$ was the determined zone associated to $H$ above, $\psi(x(t + C_k)) = (i_j)_{1 \leq j \leq 9}$, and $j$ is taken to be $6$, $7$, $8$, or $9$ based on whether $t + C_{k+g+m}$ is reached by moving up-left, up-right, down-left, or down-right from $H$. As in Stage 4, no $i_j$ determines patterns on two different holes $H$.


Now, $s_6$ is just the concatenation of these patterns on holes. Define $V_6 = U_5 \sqcup S_6 = \mathbb{Z}^2$ and 
$v_6 = u_5 s_6$. Exactly as in Stage 4, $v_6$ contains no patterns from $\mathcal{F}$, since $u_5$ was in $L(Y)$ and holes are separated by distances of at least $g$. Then $v_6 \in Y$, and so we define $\phi(x) = v_6$.\\

Finally, we must show that $\phi$ is shift-commuting, continuous, and surjective. For shift-commuting and continuity, we claim that $\phi$ is a sliding block code. To see this, we first note that the status of any site $t$ (meaning either its assigned symbol or the fact that no symbol has been assigned) after Stage 1 clearly depends only on whether $t$ is within distance $g$ from a determined zone, which is determined by knowledge of $x$ on sites within distance $k+3g+2m$ from $t$. For any subsequent stage $i$, the status of any site $t$ depends only on the status of sites after stage $i - 1$ within distance $k+3g+2m$ of $t$. Therefore,  $\phi$ is a sliding block code with radius $6(k+3g+2m)$.

The proof that $\phi$ is surjective is quite similar to the ones from \cite{BPS} and \cite{desai}, and so we only outline some slight differences here. Firstly, we only consider $x \in X$ consisting of a lattice of aligned overlapping surrounding frames as in the left-hand side of 
Figure~\ref{TSSMpic7} showing that their $\phi$-images already cover all of $Y$. In that figure, the right-hand side displays the regions of $\phi(x)$, partitioned (by color) by the stage which determined their values. However, since $\psi$ was a surjection and each $i_j$ from any $\psi(x(t + C_k))$ is used at most once, it's clear that for any $y \in Y$, the surrounded patterns $W_i$ on the left can be chosen to yield the desired subpatterns of $y$ on the right, and so $\phi$ is a surjective factor map. This completes the proof of Theorem~\ref{mainthm} for $d=2$. 

\begin{figure}[ht]
\centering
\includegraphics[scale=0.65]{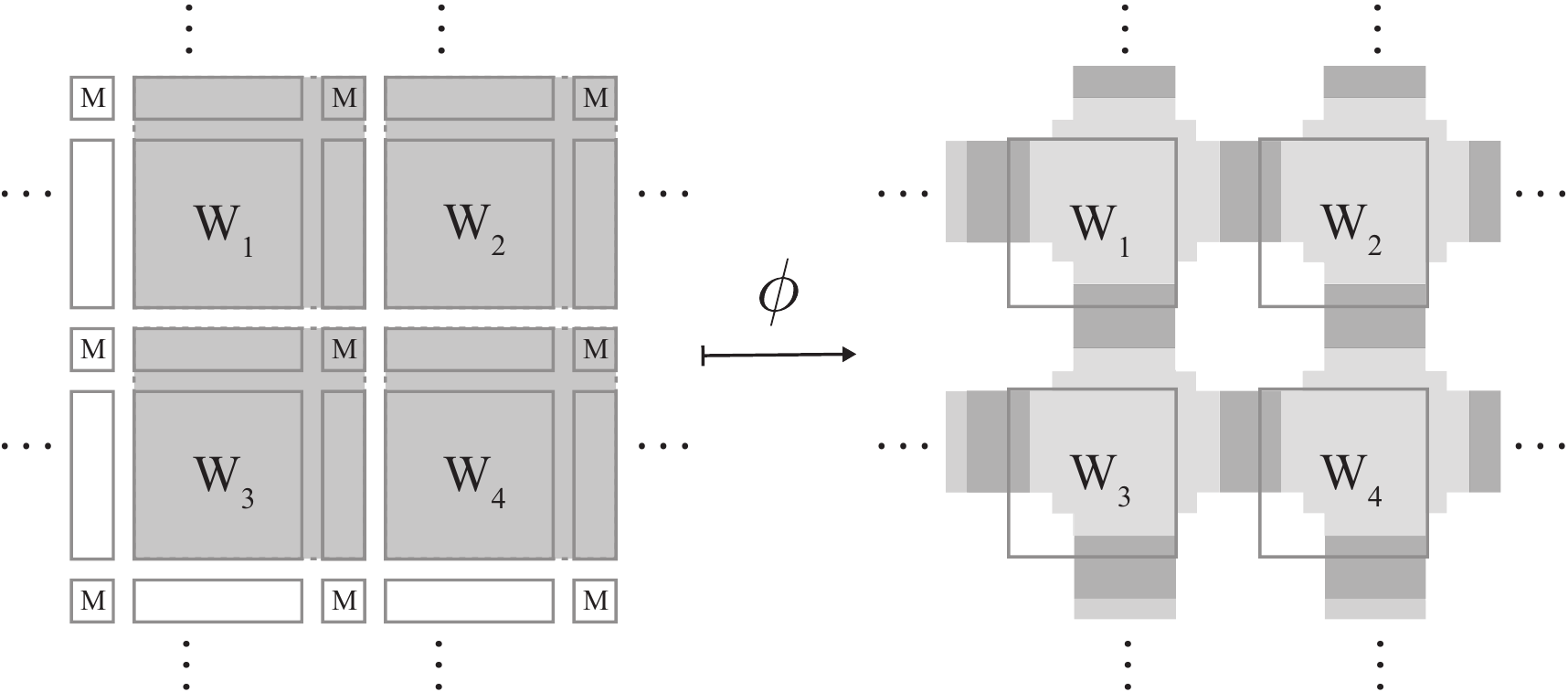}
\caption{An element of $X$ and its image under $\phi$}
\label{TSSMpic7}
\end{figure}

It remains only to describe necessary changes in the proof for $d > 2$. Markers are constructed exactly as before, with shapes which are 
$d$-dimensional hypercubes rather than squares. We choose $p > (2d+1)g$, yielding a version of (\ref{Eqn:Banff}) guaranteeing distance more than 
$(2d+3)g$ between all non-adjacent determined zones. The surjection $\psi$ for $d = 2$ had nine coordinates; one for the determined zones themselves, four for the rectangles placed in Stage 4 along edges, and four for the holes placed in Stage 6 near corners. For $d > 2$, $\psi$ has $3^d$ coordinates, again corresponding to the main bulk of a determined zone plus all its lower-dimensional ``faces.'' This requires a version of (\ref{d2bound}) in which $12g(k-3g+m) + 196g^2$ is replaced by a more complicated polynomial expression $f_d(k,g,m)$ dependent on sizes of the sets $S_i$ (defined below), and bounded from above by $d((k+3g+m+2dg)^d-(k+g+m-2dg)^d)$ ($d$ times the volume difference of two $d$-dimensional hypercubes). This polynomial has degree $d-1$ in $k$, thus the desired inequality still holds for large enough $k$ by definition of entropy.


The definition of $\phi$ proceeds in alternating stages exactly as before; for arbitrary $d$ there will be $2(d+1)$ stages. Again $U_1$ consists of sites which are at distance more than $g$ from all determined zones, and $u_1 = *^{U_1}$. Similarly, $S_2$ consists of the union of all determined zones, and $s_2$ is determined on each determined zone by knowledge of the corresponding surrounded pattern in $x$. Then, for each $j \geq 1$, $V_{2j} = U_{2j-1} \sqcup S_{2j}$ and $U_{2j+1} = V_{2j} \setminus \partial_g V_{2j}$, and so we must only describe the sets $S_{2j}$. For $1 < j \leq d+1$, $S_{2j}$ consists of all sites $t$ with the following properties:
\begin{itemize}
\item $t$ is within distance $jg$ of some determined zone in an island $I$,
\item for all $i < j$, $i$ coordinates of $t$ are within distance $(j-2+i)g$ of the corresponding $T_i(I)^c$, and
\item $d - j + 1$ coordinates of $t$ have distance more than $(2j-2)g$ from the corresponding $T_i(I)^c$.\\
\end{itemize}

We leave it to the reader to check that with this definition, each $S_{2j}$ is disjoint from $U_{2j-1}$, and $V_{2(d+1)} = \mathbb{Z}^d$. The proof that $\phi$ is a  factor map is analogous to the $d = 2$ proof, and the proof that $\phi$ is surjective simply uses $d$-dimensional versions of the points in Figure~\ref{TSSMpic7} (see \cite{BPS} and \cite{desai}); we again leave the details to the reader. 

\bibliographystyle{plain}
\bibliography{TSSM}

\end{document}